\documentclass[a4paper, 12pt]{article}
\usepackage{graphicx}
\usepackage{tikz}
\usepackage{tikz-cd}
\usepackage{xstring}
\usepackage[normalem]{ulem}
\usepackage{cancel}
\usetikzlibrary{matrix,arrows,decorations.pathmorphing}
\tikzset{commutative diagrams/.cd}
\usepackage{pst-node}

\usepackage[english]{babel}
\usepackage{geometry}
\usepackage{amssymb,amsmath,amsfonts,bbm,pifont,upgreek,bbold,accents}  
\usepackage[colorlinks=true]{hyperref}
\hypersetup{urlcolor=blue, citecolor=red}

\font\teneufm=eufm10
\font\seveneufm=eufm7
\font\fiveeufm=eufm5
\newfam\eufmfam
\textfont\eufmfam=\teneufm
\scriptfont\eufmfam=\seveneufm
\scriptscriptfont\eufmfam=\fiveeufm

\newcommand\beq[1]{\begin{equation}\label{#1} }
\newcommand{\eeq}{\end{equation} }

\newcommand\beqa[1]{\begin{eqnarray} \label{#1}}
\newcommand{\eeqa}{\end{eqnarray} }
\newcommand{\beqano}{\begin{eqnarray*} }
\newcommand{\eeqano}{\end{eqnarray*} }

\newtheorem{theorem}{Theorem}
\newtheorem{example}[theorem]{Example}
\newtheorem{definition}[theorem]{Definition}
\newtheorem{proposition}[theorem]{Proposition}
\newtheorem{lemma}[theorem]{Lemma}
\newtheorem{sublemma}[theorem]{Sublemma}
\newtheorem{remark}[theorem]{Remark}
\newtheorem{notationalremark}[theorem]{Notational Remark}
\newtheorem{corollary}[theorem]{Corollary}
\newtheorem{assumption}[theorem]{Assumption}
\newtheorem{claim}[theorem]{Claim}
\newtheorem{tools}{$\negsp\negsp$}[subsection]

\newcommand\thm[1]{\begin{theorem}\label{#1}}
\newcommand\thmtwo[2]{\begin{theorem}[#1]\label{#2}}
\newcommand\ethm{\end{theorem} }
\newcommand\dfn[1]{\begin{definition}\label{#1} \rm}
\newcommand\dfntwo[2]{\begin{definition}[#1]\label{#2} \rm}
\newcommand\edfn{\end{definition} }
\newcommand\pro[1]{\begin{proposition}\label{#1}}
\newcommand\protwo[2]{\begin{proposition}[#1]\label{#2}}
\newcommand\epro{\end{proposition} }
\newcommand\lem[1]{\begin{lemma}\label{#1}}
\newcommand\lemtwo[2]{\begin{lemma}[#1]\label{#2}}
\newcommand\elem{\end{lemma} }
\newcommand\sublem[1]{\begin{sublemma}\label{#1}}
\newcommand\sublemtwo[2]{\begin{sublemma}[#1]\label{#2}}
\newcommand\esublem{\end{sublemma} }
\newcommand\rem[1]{\begin{remark}\label{#1} \rm}
\newcommand\erem{\end{remark} }
\newcommand\notrem[1]{\begin{notationalremark}\label{#1} \rm}
\newcommand\enotrem{\end{notationalremark} }
\newcommand\cor[1]{\begin{corollary}\label{#1}}
\newcommand\cortwo[2]{\begin{corollary}[#1]\label{#2}}
\newcommand\ecor{\end{corollary} }
\newcommand\asmp[1]{\begin{assumption}\label{#1}}
\newcommand\asmptwo[2]{\begin{assumption}[#1]\label{#2}}
\newcommand\easmp{\end{assumption} }
\newcommand\clm[1]{\begin{claim}\label{#1}}
\newcommand\eclm{\end{claim} }

\expandafter\chardef\csname pre amssym.def
at\endcsname=\the\catcode`\@
\catcode`\@=11
\def\undefine#1{\let#1\undefined}
\def\newsymbol#1#2#3#4#5{\let\next@\relax
 \ifnum#2=\@ne\let\next@\msafam@\else
 \ifnum#2=\tw@\let\next@\msbfam@\fi\fi
 \mathchardef#1="#3\next@#4#5}
\def\mathhexbox@#1#2#3{\relax
 \ifmmode\mathpalette{}{\m@th\mathchar"#1#2#3}%
 \else\leavevmode\hbox{$\m@th\mathchar"#1#2#3$}\fi}
\def\hexnumber@#1{\ifcase#1 0\or 1\or 2\or 3\or 4\or 5\or 6\or 7\or
8\or
 9\or A\or B\or C\or D\or E\or F\fi}
\ifcase\@ptsize
 \font\tenmsb=msbm10
 \font\sevenmsb=msbm7
 \font\fivemsb=msbm5
\or
 \font\tenmsb=msbm10 scaled \magstephalf
 \font\sevenmsb=msbm7 scaled \magstephalf
 \font\fivemsb=msbm5  scaled \magstephalf
\or
 \font\tenmsb=msbm10 scaled \magstep1
 \font\sevenmsb=msbm7 scaled \magstep1
 \font\fivemsb=msbm5 scaled \magstep1
\fi
\newfam\msbfam
\textfont\msbfam=\tenmsb
\scriptfont\msbfam=\sevenmsb
\scriptscriptfont\msbfam=\fivemsb
\edef\msbfam@{\hexnumber@\msbfam}
\def\Bbb#1{\fam\msbfam\relax#1}
\def\widehat#1{\setboxz@h{$\m@th#1$}%
 \ifdim\wdz@>\tw@ em\mathaccent"0\msbfam@5B{#1}%
 \else\mathaccent"0362{#1}\fi}
\def\widetilde#1{\setboxz@h{$\m@th#1$}%
 \ifdim\wdz@>\tw@ em\mathaccent"0\msbfam@5D{#1}%
 \else\mathaccent"0365{#1}\fi}

\def\RIfM@{\relax\ifmmode}
\def\nonmatherr@#1{\errmessage{\string#1\space allowed only in math mode}}
\def\Bbb{\RIfM@\expandafter\Bbb@\else
 \expandafter\nonmatherr@\expandafter\Bbb\fi}
\def\Bbb@#1{{\Bbb@@{#1}}}
\def\Bbb@@#1{\fam\msbfam\relax#1}
\def\setboxz@h{\setbox\z@\hbox}
\def\wdz@{\wd\z@}
\catcode`\@=\csname pre amssym.def at\endcsname
\newcommand{\negsp}{\hspace{-.09truecm}}  

\newcommand{\integer}{{\Bbb Z}   }
\newcommand{\complex}{{\Bbb C}   }

\newcommand{{\cH}}{{\cal H} }

\newcommand{\Pic}{\rm Pic}

 \title{\bf Some Coble surfaces admit only Bertini involutions}
 \date{}
 \author{Federico Pieroni}

\begin{document}
\maketitle

\begin{abstract}
The goal of this article is to prove the main Theorem \ref{main}. It states that if $X$ is an unnodal Coble surface, and the boundary curve of $X$ is irreducible, then an involution $i : X  \to X$ is necessarily a lift of a Bertini involution.\\ We then proceed to show with some counterexamples that the hypothesis of irreducibility of $C$ and unnodality of $X$ are both necessary for this claim. {\bf MSC (2020)}: 14J50, 14J26 (Primary); 14C20 (Secondary).
\end{abstract}

\tableofcontents

\newpage
\section*{Introduction}
\addcontentsline{toc}{section}{Introduction}
In this article, a surface $X$ is a smooth, irreducible, projective variety of dimension $2$ over the complex numbers $\complex$. An involution on a surface $X$ is a regular, non - identical morphism $i : X \to X$ with $i^2 = \mathbb{1}_X.$ A surface $X$ with canonical divisor $K_X$ is called a Coble surface if it is rational, the linear system $|- K_X|$ is empty, while $|- 2 K_X|$ contains exactly one effective divisor $C$, called the boundary curve of $X$. Without loss of generality, $C$ can be assumed to be smooth, and we will actually proceed under this hypothesis.\\
In Section $1$ we recall some properties on surfaces and involutions. In particular, we need the notion of intersection product $Pic (X) \times Pic (X) \to \integer$, where $Pic (X)$ denotes the group of divisors on $X$ modulo rational equivalence. The product $D_1 D_2$ of two effective divisors $D_1, D_2 \subset X$ without common components counts the intersection points with multiplicity. If a surface $X$ contains a curve $R$ which is isomorphic to $\mathbb{P}^1$ and has self intersection $R^2 = -2$, then $X$ will be called nodal, otherwise it will be unnodal.\\
Moreover, we also need the definition of a lift of an involution. If $Y$ is a surface, and $\Sigma \subset Y$ is a finite set, we consider the blow up $X := Bl_\Sigma Y$, with blow-down map $p : X \to Y$. If $Y$ is equipped with an involution $i_Y : Y \to Y$ we want the birational involution $p \circ i_Y \circ p^{-1} : X \dasharrow X$ to be biregular. In this case, $i_X$ will be called a lift of $i_Y$. An equivalent condition for the regularity of $i_X$ is that, for any point $p \in \Sigma$, its image $i_Y (p)$ must still lie in $\Sigma$. Thus $\Sigma$ must consist of pair of points, with the two points in each pair switched by $i_Y$, plus some fixed points. Some fixed points are isolated, while the others belong to a curve $R$ which is identically fixed by $i_Y$. We will see that, if $E \subset X$ is the exceptional curve over a fixed point $p \in Y$, the action of the induced involution $i_X : X \to X$ on $E$ depends on wether $p$ is isolated or not.\\
A crucial result will be the Theorem \ref{BB} by Bayle and Beauville, which lists all minimal rational pairs $(Y, i_Y)$. By definition, these pairs are the data of a rational surface $Y$, with an involution $i_Y$ which is not a lift of any other involution.\\
Section $1$ end with a brief recall on some properties of a Coble surface. A particular focus will be given to unnodal Coble surface with irreducible boundary curve $C$. In particular, we will recall the Lemma \ref{polar}, which was proved by the author in \cite{Pieroni2026}. It states that Coble surfaces with these hypothesis admit a birational model as a quintic surface in $\mathbb{P}^3$, which we are able to describe explicitly.\\
In Section $2$ we prove the main Theorem \ref{main}. It states that an involution $i_X : X \to X$ defined over an unnodal Coble surface $X$ with irreducible boundary curve is always a lift of a Bertini involution.\\
We point out that a Coble surface can be thought as a limit of Enriques surfaces, which are not rational. Nonetheless, Enriques surfaces still admit well defined involutions, called bielliptic involutions, which share a lot of properties with the rational Bertini involutions. Dolgachev and Kondo proved in \cite{EnriquesII2025} that an involution on an unnodal Enriques surface is always a bielliptic involution.\\
The strategy for our proof will be the following: if $(Y, i_Y)$ is a minimal pair from the Theorem \ref{BB} different from the Bertini involution, we want a set $\Sigma \subset Y$ satisfying simultaneously two properties.\\
Indeed, we want $i_Y (\Sigma) = \Sigma$, in order to be able to lift $i_Y$ to a biregular involution $i_X : X \to X$ on the blow up $X := Bl_\Sigma Y$.\\
On the other hand we want $|- 2 K_X|$ to be effective, in order to ensure that $X$ is a Coble surface.\\
This forces $Y$ to contain a certain irreducible curve $\overline C$, which will be singular exactly at the points in $\Sigma$. Those points will be distributed in pairs, isolated fixed points, and non - isolated fixed points. By a corollary of the previously mentioned Lemma \ref{polar}, the non - isolated fixed points can be at most $2$.\\
Now we need a further distinction: if $Y$ does not contain isolated fixed points, then the quotient surface $Z := Y / i$ is smooth, and the quotient map $q : Y \to Z$ is a ramified double cover. In this case, we compute explicitly the linear system of the quotient curve $q(\overline C) \subset Z$, and by adjunction formula we know its algebraic genus. The fact that the fixed singularities of $\overline C$ are ``few'' forces $\Sigma$ to contain ``too many'' pairs. The two nodes of each pair are identified by $q$ to a node of $q(\overline C)$. Thus $q (\overline C)$ has ``too many'' singularities, whose number will be always strictly higher then the algebraic genus, leading to a contradiction.\\
In the other case, if $Y$ actually contains isolated fixed points, this strategy does not work, since $Z$ will be a singular surface. Then the idea is to factor the blow down map $p : X \to Y$ through an intermediate, non - minimal model $Y'$, so that we have $X \to Y' \to Y$. Then we show that $Y'$ admits another minimal model, different from $Y$ and without isolated fixed points, which has been excluded by the previous proof.\\
Finally, in Section $3$ we show that the hypothesis of irreducibility of $C$ and unnodality of $X$ are both necessary, giving counterexamples of what can happen when one of the two is removed.\\
This work is part of the author's Ph.D. Thesis \cite{Pieroni2025}.

\phantomsection
\addcontentsline{toc}{section}{Acknowledgements}
\paragraph{Acknowledgements}
The author is deeply grateful and thankful to Prof. I. Dolgachev and to Prof. A. Verra, for many inspiring discussions and precious advices.


\section{The setting}
This section will be dedicated to recall some definitions and properties of involutions and rational surfaces which we will use later.
\subsection{The Picard lattice}
Let $X$ be a smooth surface. We denote by $Pic (X)$ the Picard lattice of $X$, consisting of isomorphism classes of line bundles on $X$, with the operation of the tensor product. It is isomorphic to the group of divisors on $X$ modulo rational equivalence. The canonical divisor of $X$ is denoted by $K_X \in Pic (X)$. Moreover, $Pic (X)$ admits an inner product $$Pic (X) \times Pic (X) \to \integer,$$ called the intersection product on $X$. If $D_1, D_2$ are effective curves on $X$ without common components, the product $D_1 D_2$ counts the intersection points with multiplicity, and thus $D_1 D_2 \ge 0$. However, the square $D_1^2$, called the self intersection of $X$, can assume a priori any possible value in $\integer$.\\
An irreducible curve $C \subset X$ satisfies the adjunction formula $$C^2 + C K_X = 2 p_a (C) - 2,$$ where $p_a (C)$ denotes the arithmetic genus of $C$.\\
If $f : X \to Y$ is a regular morphism between smooth surfaces, then there exist well defined induce group morphisms of push-forward: $$f_* : Pic (X) \to Pic (Y),$$ and pullback: $$f^* : Pic (Y) \to Pic (X).$$ They satisfy $$f_* (f^* (D)) = ({\rm deg\,} f)\ D \quad \forall D \in Pic (Y),$$ and $$f^* (D_1) f^* (D_2) =  ({\rm deg\,} f)\ D_1 D_2 \quad \forall D_1, D_2 \in Pic (Y),.$$
Assume that $p_1, \ldots, p_r$ are points in a smooth surface $Y$, and that $p : X \to Y$ is the blow up of $Y$ at these points, with exceptional divisors $E_1, \ldots, E_r \subset X$. Then $p^* : Pic (Y) \to Pic (X)$ is injective, and there exists an orthogonal decomposition: $$Pic (X) = p^* Pic (Y) \oplus \bigoplus_{i = 1}^r \integer E_i.$$ With these decomposition, the canonicals $K_X, K_Y$ are related by the formula $$K_X = p^* K_Y + E_1 + \cdots + E_r.$$
We will say that $X$ is nodal if there exists a curve $R \subset X$, with $R^2 = -2$ and $R \simeq \mathbb{P}^1$, otherwise $X$ will be called unnodal.
\subsection{Hirzebruch surfaces}
Consider the rank $2$ vector bundle over $\mathbb{P}^1$ given by $${\cal E} := {\cal O}_{\mathbb{P}^1} \oplus {\cal O}_{\mathbb{P}^1} (- n).$$ The $n$-th Hirzebruch surface $\mathbb{F}_n$ is by definition the projectivized $$\mathbb{F}_n := \mathbb{P} ({\cal E}).$$ The natural projection ${\cal E} \to \mathbb{P}^1$ induces a regular map $\mathbb{F}_n \to \mathbb{P}^1$, whose fibers are all isomorphic to $\mathbb{P}^1$. These fibers belong to the same linear system $F \in Pic (\mathbb{F}_n)$. \\
For $n \ge 1$, there exist a distinguished curve $$C_{-n} \subset \mathbb{F}_n.$$ This curve is isomorphic to $\mathbb{P}^1$, and it is a section of the fibration $|F|: \mathbb{F}_n \to \mathbb{P}^1$. It has self intersection $$C_{-n}^2 = -n,$$ and thus it cannot move in the surface. Moreover, it is the only irreducible curve in $\mathbb{F}_n$ with strictly negative self intersection. The curves $F$ and $C_{-n}$ generate the Picard lattice of $\mathbb{F}_n$, with intersection rules $F^2 = 0$ (since two general fibers do not meet) and $F C_{-n} = 1$ (since $C_{- n}$ is a section of the fibration).\\
When $n = 0$, we have $\mathbb{F}_0 = \mathbb{P}^1 \times \mathbb{P}^1$, which of course does not contain curves of negative self intersection. It is the only Hirzebruch surface which admits two distinct rulings over $\mathbb{P}^1$, namely the projections on the two coordinates.\\
The surface $\mathbb{F}_1$ is isomorphic to the blow up of $\mathbb{P}^2$ at a point $p$. The negative curve $C_{-1}$ is the exceptional curve over $p$, and the fibration $|F|$ is made up with the strict transforms of the lines through $p$.\\
One can prove also the converse: a smooth surface $X$ admitting a fibration $|F| : X \to \mathbb{P}^1$ such that all the fibers of $|F|$ are isomorphic to $\mathbb{P}^1$ is necessarily one of the Hirzebruch $\mathbb{F}_n$, for some $n \ge 0$.

\subsection{Involutions and double covers}
Let $Y$ be a rational surface, with an involution $i : Y \to Y$. A Theorem by Zhang (see \cite{ZhangDolgachev2001}) states that the fixed locus $Fix (i) := \{p \in Y {\rm\ such\, that\,}\ i (p) = p\}$ is smooth, thus it is the disjoint union of smooth curves and isolated fixed points.\\
The quotient surface $Z := Y / i$ is rational, and it is smooth if and only if $Y$ does not contain isolated fixed points. If $p \in Y$ is an isolated fixed point, and $q : Y \to Z$ is the quotient map, the point $q(p)$ is called an ordinary double point, that is, it is locally isomorphic to the affine cone $\{x^2 + y^2 + z^2 = 0\} \subset \mathbb{A}^3$.\\
If $Fix (i)$ has pure dimension $1$, then $Z$ is a smooth rational surface, and $q : Y \to Z$ is a double cover, ramified along the branch locus $B := q (fix (i))$.\\
Conversely, let $Z$ be any smooth surface, and let $$B \subset Z$$ be an effective, smooth, even divisor, so that there exists $L \in Pic (Z)$ such that $$B \simeq 2 L.$$ Then one can construct the double cover $$q : Y \to Z$$ branched over $B$. By construction, the surface $Y$ is smooth and it admits a regular involution $$i : Y \to Y,$$ which switches two points in each fiber of $q$. The fixed locus of $i$ is the ramification locus $R \subset Y$ of $q$. The map $q$ restricts to an isomorphism $q : R \simeq B$, and it satisfies $$q^* B = 2 R.$$ The canonical formula for $Y$ is $$K_Y = q^* K_Z + R.$$ As divisors, the relation $q_* (q^* D) = 2 D$ holds for all $D \in Pic (Z)$. Since rational surfaces have torsion-free Picard lattice, it follows that if $Z$ is rational then $q^*$ is injective.
\subsubsection{Lifted and minimal involutions}
\begin{definition}
A rational pair $(X, i_X)$ is the data of a smooth rational surface $X$, with a non - trivial involution $i_X : X \to X,\quad i_X^2 = \mathbb{1}.$
\end{definition}
Assume $(Y, i_Y)$ is a rational pair, and that $p_1, p_2 \in Y$ are distinct points with $i_Y (p_1) = p_2$. Let $p : X \to Y$ be the blow up $Bl_{p_1, p_2} Y$ at these two points, with exceptional divisors $E_1, E_2 \subset X$. We know that the differential map $d\ i_Y$ induces isomorphisms between the tangent spaces $T_{p_1} Y, T_{p_2} Y$, and that there are natural identifications $E_i \simeq \mathbb{P} (T_{p_i} Y)$. Consequently, there exists a well defined involution $i_X : X \to X$ which makes the following diagram commute:
\begin{eqnarray}\label{diag}\begin{tikzcd}
X  \arrow{d}{p} \arrow[swap]{r}{i_X} & X \arrow{d}{p}
 \\ Y \arrow{r}{i_Y} & Y
\end{tikzcd}
\end{eqnarray}
If $p \in Y$ is fixed by the involution $i_Y$, the blow-up $X := Bl_p Y$ admits an involution $i_X : X \to X$ which preserves the exceptional curve $E \subset X$ over $y.$ The blow-down map $p : X \to Y$ still makes the diagram \ref{diag} commute, and the restriction of $i_X$ on $E$ is induced by the action of the differential $d_p i_Y$. Now we need a distinction.\\
If $p$ belongs to a curve $R$ of fixed points, there are local coordinates $(s , t)$ around $p$ where $i$ takes the form $$i_Y (s, t) = (s, - t),$$ and $t = 0$ is a local equation for the fixed curve $R$. In this case, the differential $d_p i_Y$ has different eigenvalues $+ 1, -1$ on $T_p Y$, and thus the restriction of $i_X$ on $E$ is non - trivial.\\
Conversely, in a neighbourhood of an isolated fixed point $p$ there are local coordinates $(s, t)$ such that $i_Y$ can be written as $$i_Y (s, t) = (- s, - t).$$ In this case the differential $d_p i_Y$ acts as $- \mathbb{1}$ on the tangent space $T_p Y$, and thus the induced action on the exceptional divisor is $\mathbb{P} (- \mathbb{1}) = \mathbb{1}_E$.\\
These constructions can be generalized by the following definition:
\begin{definition}
An equivariant morphism $p: (X, i_X) \to (Y, i_Y)$ is a regular, birational map $p : X \to Y$ which fits in a commutative diagram:
\[\begin{tikzcd}
X  \arrow{d}{p} \arrow[swap]{r}{i_X} & X \arrow{d}{p}
 \\ Y \arrow{r}{i_Y} & Y
\end{tikzcd}
\]
We also say that $(X, i_X)$ is a lift of $(Y, i_Y)$.\\
We say that $(X, i_X)$ is minimal if every equivariant morphism $p: (X, i_X) \to (Y, i_Y)$ is an isomorphism.
\end{definition}
The next Proposition gives a criterion to check the minimality of rational pairs.
\begin{proposition}\label{check}\cite{Bayle2000}
A rational pair $(Y, i_Y)$ is not minimal if and only if there exists a $(-1)$-curve $E \subset Y$ satisfying either $E = i(E)$ or $E\ i (E) = 0$.
\end{proposition}
The next result is a Theorem by Bayle and Beauville which classifies all the minimal rational pairs.
\begin{theorem}\cite{Bayle2000}\label{BB}
If $(Y, i_Y)$ is a minimal rational pair, then one of the following statements is true:\\
$i)$ $Y$ is a Hirzebruch surface, and there exists a commutative diagram: 
\[\begin{tikzcd}
Y  \arrow{d}{} \arrow[swap]{r}{i_Y} & Y \arrow{d}{}
 \\\mathbb{P}^1 \arrow{r}{\tau} & \mathbb{P}^1
\end{tikzcd}
\]
where $\tau$ is a non-trivial involution of $\mathbb{P}^1$.\\
$ii)$ $Y$ admits a basepoint - free pencil $|F|: Y \to \mathbb{P}^1$ with the following properties. Each fiber of $|F|$ is preserved by $i_Y$, and the smooth fibers are copies of $\mathbb{P}^1$. Every singular fiber is the union of two $(-1)$-curves meeting at one point, and switched by $i_Y$. The fixed locus is a smooth bisection of $|F|: Y \to \mathbb{P}^1$.\\
$iii)$ $Y = \mathbb{P}^1 \times \mathbb{P}^1$, and $i_Y$ switches the two coordinates.\\
$iv)$ $Y = \mathbb{P}^2$, and $i_Y$ is a projective involution.\\
$v)$ $Y$ is a Del Pezzo surface of degree $2$, and $i_Y$ is the Geiser involution.\\
$vi)$ $Y$ is a Del Pezzo surface of degree $1$, and $i_Y$ is the Bertini involution.\\
Viceversa, a pair $(Y, i_Y)$ built as in case $i), \ldots, vi)$ is actually minimal, except the following ones:\\
Case i), when $Y = \mathbb{F}_1$ is the first Hirzebruch surface, or\\
Case ii), when $Y = \mathbb{F}_1$, or $Y = Bl_{p_1, p_2, p_3} \mathbb{P}^2$, with $p_i$'s non collinear points, and $i_Y$ is either the De - Jonquieres involution of degree $2$, or the quadro-quadric involution centered at $p_1, p_2, p_3$.
\end{theorem}
Of course, the non - minimal exceptions in the last statement of the Theorem \ref{BB} still admit minimal models, which appear in the list $i), \ldots, vi)$. In particular, the minimal model of an involution on $\mathbb{F}_1$ is a projective involution of case $iv)$, the minimal model of the De - Jonquieres involution of degree $2$ is $(\mathbb{P}^1 \times \mathbb{P}^1, (a, b) \to (b, a))$ of case $iii)$, while the minimal model of the quadro - quadric involution of $\mathbb{P}^2$ is $(\mathbb{P}^1 \times \mathbb{P}^1, ([X, Y], [Z, W] \to [X, - Y], [Z, - W])),$ which is a particular situation of case $i)$.

\subsection{Coble surfaces}
A smooth complex surface $X$ is a Coble surface if it is rational, and it satisfies $$|- K_X| = \emptyset,\quad |- 2 K_X| = \{C\}.$$ The curve $C \subset X$ is called the boundary curve of $X$, and it can admit a reducible decomposition $$C = C_1 + \cdots + C_r.$$ Without loss of generality, one can always assume that $C_i C_j = \emptyset$ for all $i \ne j$. In this case, the $C_i$'s are smooth rational curves, with $$C_i^2 = -4,$$ and $$K_X^2 = -r.$$
Now we state some properties of Coble surfaces which will be used in what follows.\\
The next result was proved by the author in \cite{Pieroni2026}.
\begin{lemma}\label{polar}
Let $X$ be an unnodal Coble surface, with irreducible boundary curve $C \in |- 2 K_X|$, and let $E_1, E_2, E_3 \subset X$ be three pairwise disjoint $(-1)$-curves. Then the linear system $|H| := |C + E_1 + E_2 + E_3|$ defines a regular morphism $$|H| : X \to \mathbb{P}^3,$$ which is birational over a quintic surface $\overline X$. This surface contains a tetrahedron $T \subset \overline X$ with the following properties. A vertex $p_0$ of $T$ is a triple point for $\overline X$, and the three edges of $T$ adjacent to $p_0$ are double edges for $\overline X$. The other three edges form a plane triangle, their generic point is smooth for $\overline X$, and they are the isomorphic images of the curves $E_1, E_2, E_3$.
\end{lemma}
\begin{proposition}\label{negcurves}\cite{EnriquesII2025}
Let $X$ be a Coble surface, with boundary curve $C = C_1 + \cdots + C_r \in |- 2 K_X|$. Assume $E \subset X$ is an irreducible curve with $E^2 < 0$, and that $E$ is not one of the $C_i$'s. Then $E$ is smooth, rational, and $E^2 \in \{-1, -2\}$. 
\end{proposition}

\begin{proposition}\label{move}
Let $X$ be a Coble surface with irreducible boundary curve $C$, and let $E \subset X$ be a $(-1)$-curve. Then the linear system $|C + 2 E|$ satisfies $$h^0 ({\cal O}_X (C + 2 E)) = 2.$$
\end{proposition}
${\it Proof:}$ Of course we have $$h^0 ({\cal O}_X (2 E)) = 1.$$ Moreover, by Serre's duality: $$h^2 ({\cal O}_X (2 E)) = h^0 ({\cal O}_X (K_X - 2 E)) = 0,$$ since $$|K_X - 2 E| \subset |K_X| = \emptyset$$ by the rationality of $X$. By Riemann - Roch Theorem, $$h^0 ({\cal O}_X (2 E)) - h^0 ({\cal O}_X (2 E)) + h^2 ({\cal O}_X (2 E)) = 1,$$ which gives \begin{eqnarray} \label{nohuno} h^1 ({\cal O}_X (2 E)) = 0.\end{eqnarray} Now we consider the short exact sequence $$0 \to {\cal O}_X (2 E) \to {\cal O}_X (C + 2 E) \to {\cal O}_C (C + 2 E) \to 0.$$ By equality \eqref{nohuno}, this sequence remains exact on global sections: $$0 \to H^0 ({\cal O}_X (2 E)) \to H^0 ({\cal O}_X (C + 2 E)) \to H^0 ({\cal O}_C (C + 2 E)) \to 0.$$ Since the last term equals $H^0 ({\cal O}_{\mathbb{P}^1}) = \mathbb{C}$, we find $$H^0 ({\cal O}_X (C + 2 E)) = 2.$$ $\square$

\begin{proposition}\label{blowdown}
Let $X$ be an unnodal Coble surface with irreducible boundary curve $C$, and let $p : X \to Y$ be a regular birational morphism on a smooth rational surface $Y$. Then $p$ is the contraction of a finite number of disjoint $(-1)$-curves $E_1, \ldots, E_r$ on points $p_1, \ldots, p_r \in Y$. 
\end{proposition}
${\it Proof:}$  A regular, birational morphism between smooth surfaces $X, Y$ is a composition of blow - down maps, see \cite{Shafarevich2013}. Then there exists a composition $$p : X_0 \to X_1 \to \cdots \to X_r,$$ where $$X_0 := X,\quad X_r := Y$$ and $p_i : X_{i - 1} \to X_i$ is the blow - down of a $(-1)$-curve $E_i \subset X_{i - 1}$ on a single point $p_i \in X_i$, with $$p = p_r \circ \cdots \circ p_1.$$ By the Proposition \ref{negcurves} and the unnodality of $X$, the strict transform of each $E_i$ in $X$ is either a $(-1)$-curve, or the boundary curve $C$. We first want to prove that $C$ is not the strict transform of any $E_i$.\\
Let $E_1 \subset X$ be the exceptional locus of $p_1$. If $C$ was the strict transform of some $E_i$, then both $C$ and $E_1$ would be contracted by $p$. But by the Proposition \ref{move}, the linear system $|C + 2 E_1|$ can move, so $p$ should contract all of its members, and thus $p$ could not be birational.\\
This shows that the strict transform of each $E_i$ in $X$ has still self - intersection $-1$, and thus none of the points $p_i \in X_i$ is infinitely near to $p_j \in X_j$, for any couple $i \ne j$. Thus we can think of the curves $E_i$'s as disjoint $(-1)$-curves in $X$. $\square$

\section{The main Theorem}
The purpose of this section is to prove the following result:
\begin{theorem}\label{main}
If $X$ is an unnodal Coble surface with irreducible boundary curve $\{C\} = |- 2 K_X|$, then an involution $i : X \to X$ is necessarily a lift of a Bertini involution
\end{theorem}
We need several steps to prove the previous Theorem. The first one is the following Remark.
\begin{remark}\label{nonmin}
We point out that none of the surfaces listed in the Theorem \ref{BB} is a Coble surface with irreducible boundary. Indeed, Del Pezzo surfaces, Hirzebruch surfaces, $\mathbb{P}^1 \times \mathbb{P}^1, \mathbb{P}^2$ have all effective anti - canonical class, so this excludes all the cases except at most the second one. But a surface as in case $ii)$ arises as a double cover $q : Y \to \mathbb{F}_n$ of a Hirzebruch surface $\mathbb{F}_n$ for some $n \ge 0$, branched along an hyperelliptic curve $B \subset \mathbb{F}_n$ which is a smooth bisection of the standard fibration $\mathbb{F}_n \to \mathbb{P}^1$. Since a Coble surface $X$ with irreducible boundary satisfies $K_X^2 = -1$, it suffices to prove that for a surface $Y$ as above the square $K_Y^2$ is even.\\ 
The canonical formula gives $$K_Y = q^* K_{\mathbb{F}_n} + R,$$ where $R \subset Y$ is the fixed locus of $i_Y$, and thus $$K_Y^2 = 2 K_{\mathbb{F}_n}^2 + 2 (q^*K_{\mathbb{F}_n}) R + R^2.$$ But $$R = \frac{q^* B}{2},$$ thus $$R^2 = \frac{B^2}{2}.$$ Since $B$ is an even divisor, there exists a line bundle $L \in Pic (\mathbb{F}_n)$ such that $$B = 2 L,$$ so $$R^2 = 2 L^2.$$ This shows that $K_Y^2$ is even and it concludes the proof.
\end{remark}

The next Corollary defines the strategy which we will use to exclude that a pair $(X, i_X)$, with $X$ an unnodal Coble surface with irreducible boundary, might have a minimal model different from the Bertini involution.

\begin{corollary}\label{min}
By the Remark \ref{nonmin}, if $X$ is a Coble surface with irreducible boundary, an involution $i_X : X \to X$ is non minimal. Thus there exists a commutative diagram 
\[\begin{tikzcd}
X  \arrow{d}{p} \arrow[swap]{r}{i_X} & X \arrow{d}{p}
 \\ Y \arrow{r}{i_Y} & Y
\end{tikzcd}
\]
where $(Y, i_Y)$ is one of the pairs listed in the Theorem \ref{BB}, and $p$ is a regular, birational and non - invertible morphism. If we also require that $X$ is unnodal, by the Proposition \ref{blowdown}, $p$ is the contraction of a finite number of disjoint $(-1)$-curves $E_1, \ldots, E_r \subset X$ on points $p_1, \ldots, p_r \in Y$. The commutativity of the diagram implies that for each curve $E_i$ over $p_i$, its image $i_X (E_i)$ is still one of the curve $E_j$'s. This proves that $i_X$ preserves the set of curves $\{E_1, \ldots, E_r\}$, while $i_Y$ preserves the set $\{p_1, \ldots, p_r\}$. Now we recall the canonical formula \begin{eqnarray}\label{can} K_X = p^* K_Y + E_1 + \cdots + E_r. \end{eqnarray} The class of the boundary curve $C \subset X$ becomes then: $$C = - 2 K_X = p^* (-2 K_Y) - 2 E_1 \cdots - 2 E_r,$$ which means that there exists a curve $$\overline C := p(C) \in |- 2 K_Y|$$ which is singular at $p_1, \ldots, p_r$. Since $C$ is irreducible, the curve $\overline C$ is irreducible too.\\ Of course in $Pic (X)$ the relation $$i_X (K_X) = K_X$$ holds, and thus $$i_X (C) = C.$$ By the commutativity of the diagram we find $$i_Y (\overline C) = \overline C.$$ Finally, taking squares on both sides of relation \eqref{can} we get $$r = K_Y^2 + 1,$$ so that the number of points $p_1, \ldots, p_r$ only depends on $Y$. We also note that the quantity $K_Y^2 + 1$ equals to the algebraic genus of the curve $\overline C$ in the system $|- 2 K_Y|$, coherently with the fact that the normalization $C$ is isomorphic to $\mathbb{P}^1$.
\end{corollary}
The previous Corollary estabilishes our strategy: if $(Y, i_Y)$ is a pair appearing in cases $i), \ldots, v)$ of the Theorem \ref{BB}, we are looking for an irreducible, rational curve $\overline C \subset Y$, belonging to the linear system $|- 2 K_Y|$ and invariant under $i_Y$. The singular locus $Sing (\overline C)$ has cardinality $K_Y^2 + 1$, and it has to be invariant as well. A priori, it looks we will need a lot of combinatorics, since these points will be distributed in fixed points (which might be isolated or not), and pairs of points, with $i_Y$ switching the two points of each pair. In each of these cases we will look for a contradiction.\\
The next result heavily simplifies this situation.
\begin{lemma}\label{crucial}
Let $(Y, i_Y)$ be a smooth rational pair, and let $p_1, \ldots, p_m \in Y$ be a finite set of points. Assume that the following properties are all simultaneously satisfied:\\
$i)$ The involution $i_Y$ set-theoretically preserves these points: $i_Y \{p_1, \ldots, p_m\} = \{p_1, \ldots, p_m\}$.\\
$ii)$ The blow - up $X := Bl_{p_1, \ldots, p_m} Y$ is a Coble surface.\\
$iii)$ The boundary curve $C \in |- 2 K_X|$ is irreducible.\\
$iv)$ $X$ is unnodal.\\
Then at most $2$ of the $p_i$'s are non - isolated fixed points for $i_Y$.
\end{lemma}
${\it Proof:}$ By hypothesis $i)$, there exists a well - defined lifted involution $i_X : X \to X$ induced from $i_Y$. Of course it satisfies $$i_X (K_X) = K_X,$$ and multiplying both sides by $-2$, we get: \begin{eqnarray}\label{Cpres} i_X (C) = C.\end{eqnarray} By contradiction, assume now that $p_1, p_2, p_3$ are non-isolated fixed points. Let $E_1, E_2, E_3 \subset X$ be the corresponding exceptional divisors. By hypothesis $ii), iii), iv)$ we can apply the Lemma \ref{polar}, which gives us a morphism $|H| : X \to \mathbb{P}^3$. The involution $i_X$ preserves each $E_i$, hence it preserves the polarization $|H|$ by relation \eqref{Cpres}. Thus there exists a linear projective involution $I : \mathbb{P}^3 \to \mathbb{P}^3$ which makes the following diagram commute: 
\[\begin{tikzcd}
X  \arrow{d}{{|H|}} \arrow[swap]{r}{i_X} & X \arrow{d}{{|H|}}
 \\\mathbb{P}^3 \arrow{r}{I} & \mathbb{P}^3
\end{tikzcd}
\]
The plane $\Pi_0 \subset \mathbb{P}^3$ corresponding to the section $C + E_1 + E_2 + E_3$ is preserved by $I$ by construction. By the Lemma \ref{polar}, the images of $E_1, E_2, E_3$ inside $\Pi_0$ are three lines $L_1, L_2, L_3$ which form a plane triangle. Since $i_X$ preserves all the $E_i$'s, the involution $I$ must preserve all the $L_i$'s. But then $I$ fixes the three intersection points $L_i \cap L_j$, which are the vertices of the triangle. Hence one of the $L_i$, say $L_1$, must be a fixed line. Since $L_1$ is the isomorphic image of $E_1$, this is equivalent to say that $i_X$ acts as the identity on $E_1$, and thus $p_1 \in Y$ is an isolated fixed point, a contradiction. $\square$\\\\
Now we have all the tools we need to start the proof.
\begin{lemma}\label{nothree}
Assume that $(X, i_X)$ is a rational pair, with $X$ an unnodal Coble surface with irreducible boundary curve $C$. Then an equivariant morphism $p : (X, i_X) \to (\mathbb{P}^1 \times \mathbb{P}^1, (a, b) \to (b, a))$ does not exist.
\end{lemma}
{\it Proof:} By contradiction, assume that such a morphism $p$ exists. The Corollary \ref{min} states that $\mathbb{P}^1 \times \mathbb{P}^1$ must contain a curve $\overline C \in |- 2 K_{\mathbb{P}^1 \times \mathbb{P}^1}| = |{\cal O}_{\mathbb{P}^1 \times \mathbb{P}^1} (4, 4)|$ which must be irreducible, invariant under the involution $(a, b) \to (b, a)$, and it must have the highest possible number of double points, which equals the algebraic genus $p_a (\overline C) = 9$.\\
The quotient of $\mathbb{P}^1 \times \mathbb{P}^1$ under the involution is given by the double cover $$q : \mathbb{P}^1 \times \mathbb{P}^1 \to \mathbb{P}^2,$$ $$q ([X, Y], [Z, W]) := (X Z, X W + Y Z, Y W).$$ The map $q$ is induced by the subsystem of sections of $|{\cal O}_{\mathbb{P}^1 \times \mathbb{P}^1} (1, 1)|$ which are symmetric with respect to the involution. Since $\overline C$ is invariant and it has class $|{\cal O}_{\mathbb{P}^1 \times \mathbb{P}^1} (4, 4)|$, $\overline C$ is the double cover of an irreducible quartic curve $Q \subset \mathbb{P}^2$. Since $\overline C$ is invariant under the involution, the set of the $9$ double points must be invariant too. By the Lemma \ref{crucial}, exactly one of the nodes belongs to the diagonal $\Delta \subset \mathbb{P}^1 \times \mathbb{P}^1$, while the other $8$ double points are displaced in $4$ pairs. The two points of each pair are switched by the involution. Since $q$ is a local isomorphism away from $\Delta$, it identifies the two nodes of each pair to a node of the curve $Q$. But an irreducible plane quartic curve cannot have $4$ singular points. $\square$

\begin{lemma}\label{nofour}
Assume that $(X, i_X)$ is a rational pair, with $X$ an unnodal Coble surface with irreducible boundary curve $C$. Then an equivariant morphism $p : (X, i_X) \to (\mathbb{P}^2, I)$, with $I$ a projective involution, does not exist.
\end{lemma}
{\it Proof:} By contradiction, assume that such a morphism $p$ exists. The Corollary \ref{min} states that $\mathbb{P}^2$ must contain a curve $\overline C \in |- 2 K_{\mathbb{P}^2}| = |{\cal O}_{\mathbb{P}^2} (6)|$ which must be irreducible, invariant under the involution $I$, and it must have the highest possible number of double points, which equals the algebraic genus $p_a (\overline C) = 10$.\\
Since $\overline C$ is $I$-invariant, this set of double points is invariant too. We first note that the fixed locus of $I$ consist of a line $L_0 \subset \mathbb{P}^2$, and an isolated fixed point $p_0 \in \mathbb{P}^2 \setminus L_0$. There are three possibilities: $p_0$ could be one of the $10$ double points of the sextic $\overline C$, or a smooth point of $\overline C$, or it could not lie in $\overline C$ at all.\\
In all these situations, it is impossible that all the $10$ double points of $\overline C$ are fixed by $I$, otherwise at least $9$ of them (all but at most $p_0$) should lie on $L_0$. There are two reasons why this is impossible. The first one is that we would have an intersection product $\overline C L_0 \ge 18$, which contradicts Bezout's Theorem $\overline C L_0 = 6$. The second reason is given by the Lemma \ref{crucial}, which states that actually at most two singular points of $\overline C$ can lie on $L_0$.\\
Hence there exists at least one pair of nodes $a, b \in {\rm\, Sing\,} (\overline C)$ switched by $I$. Then the map $p : X \to \mathbb{P}^2$ factors through the blow - up $X \to Bl_{a, b} \mathbb{P}^2  \to \mathbb{P}^2.$ This blow up is equipped with a well - defined involution, induced by $I$, which preserves the $(-1)$-curve consisting on the strict transform of the line ${\overline {a, b}}$. The contraction of this curve leads us to a minimal model of $(X, i_X)$ consisting of $\mathbb{P}^1 \times \mathbb{P}^1$, equipped with the involution which exchanges the two factors. This is forbidden by the Lemma \ref{nothree}. $\square$\\

\begin{lemma}\label{notwo}
Assume that $(X, i_X)$ is a rational pair, with $X$ an unnodal Coble surface with irreducible boundary curve $C$. Then an equivariant morphism $p : (X, i_X) \to (Y, i_Y)$, with $(Y, i_Y)$ a minimal pair as in case $ii)$ of the Theorem \ref{BB}, does not exist.
\end{lemma}
{\it Proof:} Assume the contrary: then we find a regular birational morphism $$p : X \to Y,$$ with $Y$ a smooth rational surface equipped with an involution $$i_Y : Y \to Y$$ such that $$p \circ i_X = i_Y \circ p.$$ Moreover, $Y$ admits a conic fibration $\pi_Y : Y \to \mathbb{P}^1$, whose fibers are all preserved by $i_Y$. We already saw in the Remark \ref{nonmin} that such an involution $i_Y$ arises as the deck involution of a double cover $q : Y \to \mathbb{F}_n$ on a Hirzebruch surface $\mathbb{F}_n$, branched over a smooth bisection $B \subset \mathbb{F}_n$. The singular fibers of $\pi_Y: Y \to \mathbb{P}^1$ correspond to the ramification locus of the restriction $B \to \mathbb{P}^1$. Putting all together, we have a diagram:
\[\begin{tikzcd}
X \arrow{r}{i_X} \arrow[swap]{d}{p} & X \arrow{d}{p} &
\\ Y  \arrow{r}{i_Y} & Y \arrow{r}{q} \arrow[swap]{dr}{\pi_Y} & \mathbb{F}_n \arrow{d}{}
\\  &  & \mathbb{P}^1
\end{tikzcd}
\]
We first want to exclude the case $n > 0$. Assume this actually happens, and consider the negative section $C_{- n} \subset \mathbb{F}_n$. If the pullback $q^* (C_{-n}) \subset Y$ is irreducible, then it is a curve of self intersection $$q^*(C_{-n})^2 = - 2 n \le -2.$$ By Proposition \ref{negcurves} and the unnodality assumption, the strict transform of $q^* (C_{-n})$ in $X$ must necessarily coincide with a component of the Coble curve $C$. Since $C$ is irreducible, we have $$p (C) = q^* (C_{- n})$$ But by the Proposition \ref{blowdown}, $Y$ is obtained by $X$ as a blow - down of $(-1)$-curves. Each $(-1)$-curve of $X$ intersects $C$ in two points, hence its contraction makes the self-intersection of $C$ jump by $4$, that is, $$p(C)^2 \ge C^2 + 4 = 0,$$ which is a contradiction.\\
Another contradiction arises if we assume that $q^* (C_{- n})$ actually splits as $$q^* (C_{- n}) = A + B$$ with $A, B$ curves switched by $i_Y$. In this case, we find the same self - intersection $$- 2 n =q^* (C_{-n})^2 = A^2 + B^2 + 2 A B = 2 A^2 + 2 A B.$$ This shows that $A, B$ are two distinct curves with negative self - intersection, hence $$A^2 = B^2 = - 1$$ by Proposition \ref{negcurves}. So the previous equality becomes $$- 2 n = 2 A B - 2$$ which is possible only if $A B = 0$ and $n = 1$. This again is a contradiction, since $i_Y$ should exchange two disjoint $(-1)$-curves, against the hypothesis of minimality.\\
The last case to exclude is that $C_{- n}$ might be a component of the branch locus $B \subset \mathbb{F}_n$ of $q$. Since $B$ is a bisection of $\mathbb{F}_n$, the other component $B - C_{- n}$ must be another section, disjoint from $C_{- n}$ (otherwise $Y$ would be a singular surface). The only curves with this property live in the linear system $C_{- n} + n F$, where $F$ is a fiber of the $\mathbb{P}^1$-bundle $|F| : \mathbb{F}_n \to \mathbb{P}^1$. Thus we find $$B = 2 C_{- n} + n F,$$ which forces $n$ to be even. Let $R \subset Y$ be the ramification component corresponding to $C_{- n}$, so that $$q^* (C_{- n}) = 2 R.$$ Taking squares, we have $$- 2 n = 4 R^2,$$ that is $$R^2 = - \frac{n}{2}$$ Now we can argue as before: if $n \ge 4$, then $R^2 \le -2$ and we still find $R = p (C)$, which is a contradiction, since $p(C)^2 \ge 0$. But if $n = 2$, we still have a contradiction, since $R^2 = - \frac{n}{2} = -1$ would give a $(-1)$-curve fixed by $i_Y$, against the assumption of minimality.\\
Finally, we have proved that $n = 0$, that is $$\mathbb{F}_n = \mathbb{F}_0 = \mathbb{P}^1 \times \mathbb{P}^1.$$ To avoid confusion, we denote by $F_1 \in \Pic (\mathbb{P}^1 \times \mathbb{P}^1)$ the ruling induced from $Y$, so that $q^* (F_1)$ is the conic bundle on $Y$ given by the Theorem \ref{BB}. Let $F_2$ be the other ruling. In this notation, the branch curve $B$ of $q$ lives in the linear system \begin{eqnarray}\label{branch} B \in |a F_1 + 2 F_2|\end{eqnarray} for some even number $a \ge 0$.\\
We now prove that the only admissible values for $a$ are $0, 2, 4$. Indeed, by the Proposition \ref{blowdown}, $X$ is a blow-up of $Y$ at a finite set, thus $$K_Y^2 \ge K_X^2 = -1.$$ Since $Y$ is rational, its Picard lattice is torsion free. Thus it makes sense to write the canonical formula as: $$K_Y = q^* K_{\mathbb{P}^1 \times \mathbb{P}^1} + \frac{1}{2} q^* B = q^* ((\frac{a}{2} - 2) F_1 - F_2).$$ Thus $$-1 \le K_Y^2 = 2 ((\frac{a}{2} - 2) F_1 - F_2)^2 = 8 - 2 a,$$ which gives $a \in \{0, 2, 4\}$, as required.\\
If $a = 0$ then $B$ is the disjoint union of two horizontal lines in $F_2$. This means that $Y = \mathbb{P}^1 \times \mathbb{P}^1$ too, with the involution $i_Y : \mathbb{P}^1 \times \mathbb{P}^1 \to \mathbb{P}^1 \times \mathbb{P}^1$ acting as the identity on the first component. In a suitable choice of coordinates $([X_0, X_1], [Y_0, Y_1])$ on $Y$, we have $$i_Y ([X_0, X_1], [Y_0, Y_1]) = ([X_0, X_1], [Y_0, - Y_1])$$ and \begin{eqnarray}\label{quoz} q([X_0, X_1], [Y_0, Y_1]) = ([X_0, X_1], [Y_0^2, Y_1^2]). \end{eqnarray} We saw in the Corollary \ref{min} that the image $\overline C := p(C)$ must live in the linear system $$\overline C \in | - 2 K_{\mathbb{P}^1 \times \mathbb{P}^1}| = |{\cal O}_{\mathbb{P}^1 \times \mathbb{P}^1}(4, 4)|.$$ An irreducible member of this system has algebraic genus $9$, as it is confirmed by the difference $$K_Y^2 - K_X^2 = 9.$$ Hence $Y$ comes from the contraction via $p$ of $9$ $(-1)$-curves of $X$ over $9$ nodes of $\overline C$, and these points must be set - theoretically preserved by $i_Y$. Denote by $s$ the number of pairs of nodes which are switched by $i_Y$, and by $t$ the number of nodes fixed by $i_Y$, with $$2 s + t = 9.$$ By the Lemma \ref{crucial}, we must have $t < 3$, so that $$s = 4, \ t = 1$$ is the only possible solution. 
Now let: $$D := q(\overline C) \subset \mathbb{P}^1 \times \mathbb{P}^1$$ be the set - theoretic image of $p(C)$ in the final $\mathbb{P}^1 \times \mathbb{P}^1$, with the reduced scheme structure. Clearly we have $$\overline C = q^* (D).$$
But the expression \ref{quoz} shows that the induced pull - back $q^*$ acts as $F_1 \to F_1, F_2 \to 2 F_2$, so that $$q^* (D) = \overline C = 4 F_1 + 4 F_2 = q^*(4 F_1 + 2 F_2),$$ and hence $$D \in |4 F_1 + 2 F_2|.$$
The unique fixed double point of $\overline C$ corresponds to a tangency point between $D$ and the branch locus $B$, while the $4$ pairs correspond to $4$ actual nodes of $D$. But an irreducible curve $D \in |4 F_1 + 2 F_2|$ has algebraic genus $3$, hence it cannot admit $4$ nodes.\\
If $a = 2$, then $q : Y \to \mathbb{P}^1 \times \mathbb{P}^1$ is branched over a curve of type $B \in |2 F_1 + 2 F_2|$, and we can compute $K_Y$ as $$K_Y = q^* (K_{\mathbb{P}^1 \times \mathbb{P}^1}) + \frac{1}{2} q^*(B) = q^*(- F_1 - F_2).$$
Note that $$K_Y^2 = 4,$$ so by the Corollary \ref{min} $Y$ is obtained from $X$ via the contraction of $K_Y^2 + 1 = 5$ $(-1)$-curves, which correspond to $5$ double points of $\overline C$. As before, assume that $s$ pairs of these points are switched by $i_Y$, while the remaining $t$ are fixed, with $$2 s + t = 5.$$ Since there are not isolated fixed points in $Y$, the Lemma \ref{crucial} gives again $t < 3$, that is $$s = 2, \ t = 1$$ necessarily. Again, this means that $q(\overline C)$ has $2$ nodes. Since $\overline C$ is $i_Y$-invariant and irreducible, we have $$\overline C = q^* D$$ for some effective irreducible divisor $D \subset \mathbb{P}^1 \times \mathbb{P}^1$. Thus: $$q^* (D) = \overline C = - 2 K_Y = q^* (2 F_1 + 2 F_2)$$ so that $$D \in |2 F_1 + 2 F_2|.$$ But an irreducible curve in this linear system has algebraic genus $1$, so it can have at most one node, contradicting $s = 2$.\\
There is only one case left, that is the number $a$ in \ref{branch} equals $4$, so that $B = 4 F_1 + 2 F_2$. Again, we compute can compute $K_Y$ as $$K_Y = q^* (K_{\mathbb{P}^1 \times \mathbb{P}^1}) + \frac{1}{2} q^*(B) = q^*(- F_2).$$ Hence the canonical has square $$K_Y^2 = 0$$ and hence $Y$ is obtained from $X$ via the contraction of only one $(-1)$-curve $E \subset X$ to a point $p \in Y$. But $|F_2|$ can move in $\mathbb{P}^1 \times \mathbb{P}^1$, and hence $|- K_Y| = |q^* (F_2)|$ can move in $Y$. But then one member of $|- K_Y|$ passes through $p$, which means that $|- p^* K_Y - E|$ is non empty. This contradicts the very Coble assumption $|- K_X| = \emptyset$. $\square$\\

\begin{lemma}\label{noone}
Assume that $(X, i_X)$ is a rational pair, with $X$ an unnodal Coble surface with irreducible boundary curve $C$. Then an equivariant morphism $p : (X, i_X) \to (\mathbb{F}_n, i')$, with $(\mathbb{F}_n, i')$ a minimal pair supported on a Hirzebruch surface $\mathbb{F_n}$, does not exist.
\end{lemma}
{\it Proof:} Assume the opposite, so that we are in case $i)$ of Theorem \ref{BB}: then there exists a commutative diagram:
\[\begin{tikzcd}
X \arrow{r}{i_X} \arrow[swap]{d}{p} & X \arrow{d}{p} 
\\ \mathbb{F}_n  \arrow{r}{i'}  \arrow[swap]{d}{f} & \mathbb{F}_n \arrow{d}{f}
\\  \mathbb{P}^1 \arrow{r}{\tau} & \mathbb{P}^1
\end{tikzcd}
\]
where $\mathbb{F}_n$ is a Hirzebruch surface, with its $\mathbb{P}^1$-bundle $f : \mathbb{F}_n \to \mathbb{P}^1$, $p$ is a regular birational morphism, and $\tau$ is a nontrivial involution on $\mathbb{P}^1$.\\
Again, we start excluding the case $n > 0$: note that $n = 1$ is impossible, since $(\mathbb{F}_1, i')$ is never a minimal pair. Hence $n \ge 2$, but then the negative curve $C_{- n} \subset \mathbb{F}_n$ has a strict transform in $X$ with self intersection lesser or equal than $-2$. Again, by Proposition \ref{negcurves}, this strict transform must coincide with the anti - bicanonical curve $C$, so that $p (C) = C_{-n}$. But this is impossible, since $p: X \to \mathbb{F}_n$ contracts $(-1)$-curves, and hence $C_{-n}$ should be singular, which is false.\\
This gives $n = 0$, so we have $$i' : \mathbb{P}^1 \times \mathbb{P}^1 \to \mathbb{P}^1 \times \mathbb{P}^1.$$ By the Theorem \ref{BB}, the involution $i'$ preserves the linear system given by the first ruling $|F_1|$, and it acts on it as the involution $\tau$ of the previous diagram. But then $i'$ must preserve also the second ruling $|F_2|$, so let denote by $\sigma$ the action of $i': \mathbb{P}^1 \to \mathbb{P}^1$ on the second coordinate of the product. We already saw during the proof of the Lemma \ref{notwo} how to deal with the case $\sigma = \mathbb{1}_{\mathbb{P}^1}$, so we can assume $\sigma \ne \mathbb{1}_{\mathbb{P}^1}$. In a suitable choice of coordinates, we can write \begin{eqnarray} \label{quasihor} i' ([X_0, X_1], [Y_0, Y_1]) = ([X_0, - X_1], [Y_0, - Y_1])\end{eqnarray} It has four fixed points, namely $([1, 0], [1, 0]), ([1, 0], [0, 1]), ([0, 1], [1, 0]), ([0, 1], [0, 1])$. Moreover, by the Corollary \ref{min}, the image $p (C)$ is an irreducible member of the linear system \begin{eqnarray} \label{quasihordue} p (C) \in |{\cal O}_{\mathbb{P}^1 \times \mathbb{P}^1} (4, 4)|\end{eqnarray} and hence it must have $9$ double points. This set of points is preserved by $i'$, hence at least one of them must coincide with one of the four fixed points. Assume for example that this node is $a = ([1, 0], [1, 0])$. Then the map $p : X \to \mathbb{P}^1 \times \mathbb{P}^1$ factors through the blow - up $X \to Bl_a (\mathbb{P}^1 \times \mathbb{P}^1) \to \mathbb{P}^1 \times \mathbb{P}^1$. The involution induced by $i'$ on this blow - up preserves both the strict transforms of the two fibers $\{X_1 = 0\}, \{Y_1 = 0\}$ containing $a$. Hence we can contract these $(-1)$-curves, and this produces another minimal model:
\[\begin{tikzcd}
X \arrow{r}{i_X} \arrow[swap]{d}{p'} & X \arrow{d}{p'} 
\\  \mathbb{P}^2 \arrow{r}{I} & \mathbb{P}^2
\end{tikzcd}
\]
where $I$ is a projective involution. This is forbidden by the Lemma \ref{nofour}. $\square$\\

\begin{lemma}\label{nofive}
Assume that $(X, i_X)$ is a rational pair, with $X$ an unnodal Coble surface with irreducible boundary curve $C$. Then an equivariant morphism $p : (X, i_X) \to (Y, i_Y)$, with $i_Y$ the Geiser involution on a Del Pezzo surface $Y$ of degree $2$, does not exist.
\end{lemma}
{\it Proof:} By the Corollary \ref{min}, $p$ is the contraction of $K_Y^2 + 1 = 3$ $(-1)$-curves $E_1, E_2, E_3$ of $X$ on three points $p_1, p_2, p_3 \in Y$, which are double points for a curve $\overline C \in |- 2 K_Y|$ preserved by the Geiser involution. By the Lemma \ref{crucial}, only one of these points, say $p_1$, is fixed by $i_Y$, while $p_2, p_3$ are switched. We recall that the Geiser involution $i_Y$ is associated to a double cover $q : Y \to \mathbb{P}^2$ branched over a smooth quartic curve, which is induced by the complete linear system $|- K_Y|$. Since $i_Y$ switches $p_2, p_3$, they have the same image $a = q(p_2) = q(p_3) \in \mathbb{P}^2$. The line through $a$ and $q (p_1)$ defines a section of $|- K_Y|$ through all the three points $p_1, p_2, p_3$. This means that $|- p^* K_Y - E_1 - E_2 - E_3|$ is non empty, against the very Coble assumption $|- K_X| = \emptyset$. $\square$\\\\
The Theorem \ref{main} is now a combined consequence of the Lemmas \ref{nothree}, \ref{nofour}, \ref{notwo}, \ref{noone}, \ref{nofive}.
\section{Counterexamples}
In this section, we will construct some counter-examples to Lemmas \ref{nothree}, \ref{nofour}, \ref{notwo}, \ref{noone}, \ref{nofive}, building Coble surfaces which either admit reducible boundary, or possess $(-2)$-curves, and which are equipped with some of the involutions forbidden by these $5$ Lemmas.\\
In all our constructions, we will need the following Remark:
\begin{remark}\label{easy}
Assume that a rational surface $X$ admits a collection of curves $C_1, \ldots, C_n \subset X$ such that $C_i C_j = - 4 \delta_{i, j}, C_i \simeq \mathbb{P}^1$ and $C_1 + \cdots + C_n = - 2 K_X$. Then $X$ is a Coble surface.\\
Indeed, we have the short exact sequence $$0 \to {\cal O}_X \to {\cal O}_X (C_1 + \cdots + C_n) \to \bigoplus_{i = 1}^n {\cal O}_{C_i} (C_1 + \cdots + C_n) \to 0.$$ By assumption the last term equals $$\bigoplus_{i = 1}^n {\cal O}_{C_i} (C_1 + \cdots + C_n) = \bigoplus_{i = 1}^n {\cal O}_{\mathbb{P}^1} (- 4),$$ which has no sections. Thus there exists an isomorphism $$H^0 ({\cal O}_X (C_1 + \cdots + C_n)) \simeq H^0 ({\cal O}_X) \simeq \mathbb{C},$$ that is $$|- 2 K_X| = \{C_1 + \cdots + C_n\}.$$ Moreover, the linear system $|- K_X|$ must be empty: an effective curve ${\cal E} \in |- K_X|$ would induce a non reduced curve $2 {\cal E} \in |- 2 K_X|$, which is impossible since $|- 2 K_X|$ cannot move by the previous proof.
\end{remark}
\begin{example}
Let $Q \subset \mathbb{P}^2$ be a smooth irreducible quartic curve, and let $C'' \subset \mathbb{P}^2$ a smooth irreducible conic. Assume that the intersection $Q \cap C$ takes the form $$Q \cap C'' = 2 q_1 + 2 q_2 + 2 q_3 + a + b,$$ with the $q_i$'s, $a, b$ five distinct points. Let $$q : Y \to \mathbb{P}^2$$ be the double cover branched over $Q$. Then $Y$ is a Del Pezzo surface, with the Geiser involution $i_Y : Y \to Y$. Since the points $a, b$ are distinct, the preimage $C' := q^* (C'')$ is irreducible, and since $C''$ is tangent at $Q$ at the points $q_i$'s, the curve $C'$ has three double points $p_1, p_2, p_3 \in Y$. Note that the $p_i$'s lie along the ramification curve $R \subset Y$. Hence the blow up $$X := Bl_{p_1, p_2, p_3} Y$$ admits a well defined involution $i_X : X \to X$, which preserves the exceptional divisors $E_1, E_2, E_3 \subset X$ over the $p_i$'s. Moreover, the surface $X$ is a Coble surface with irreducible boundary.\\
Indeed, the canonical system $K_Y$ equals $$K_Y = q^* (K_{\mathbb{P}^2}) + \frac{1}{2} q^* (Q) = q^* (- L).$$ This implies $$C' = q^* (C) = q^* (2 L) = - 2 K_Y.$$ Since the $p_i$'s are double points for $C'$, the strict transform $C$ of $C'$ has class $$C = C' - 2 E_1 - 2 E_2 - 2 E_3 = -2 (K_Y + E_1 + E_2 + E_3) = - 2 K_X.$$ Moreover, the canonical $K_Y$ has square $$K_Y^2 = q^* (- L)^2 = 2,$$ and thus $$K_X^2 = K_Y^2 - 3 = - 1.$$ This shows $$C^2 = 4 K_X^2 = -4.$$ By the Remark \ref{easy}, $X$ is a Coble surface.\\
There is no contradiction with the Lemma \ref{nofive}, since we we can exhibit $(-2)$-curves inside $X$. Let $$L_i \subset \mathbb{P}^2$$ be the lines which are tangent to the quartic $Q$ at the point $q_i$, and let $${\cal E}_i := q^* L_i \subset Y$$ be the pre image in $Y$. The double cover ${\cal E}_i \to L_i$ has $4$ branch points, counted with multiplicity, thus ${\cal E}_i$ has arithmetic genus $1$. Since $L_i$ is tangent to $Q$, its double cover ${\cal E}_i$ has a node at $p_i$, so that ${\cal E}_1$ is a rational curve. Note that $${\cal E}_i^2 = q^* (L_i)^2 = 2.$$ When we blow up the $p_i$'s to get the Coble surface $X$, the singularity of ${\cal E}_i$ at $p_i$is resolved, thus the strict transform $R_i \subset X$ is a copy of $\mathbb{P}^1$. Moreover, the class of $R_i$ equals $$R_i = {\cal E}_i - 2 E_i,$$ so that $$R_i^2 = {\cal E}_i^2 - 4 = -2.$$
\end{example}
\begin{example}
Assume that $Q \subset \mathbb{P}^2$ is an irreducible quartic curve with three nodes $q_1, q_2, q_3$, and assume that $B \subset \mathbb{P}^2$ is an irreducible conic which has three tangency points $x_1, x_2, x_3$ with $Q$, with $q_i \ne x_j$ for all indices $i, j$. Let $q : \mathbb{P}^1 \times \mathbb{P}^1 \to \mathbb{P}^2$ be the double cover branched over $B$, and let $Q' \subset \mathbb{P}^1 \times \mathbb{P}^1$ be the pre-image of $Q$. The curve $Q'$ belongs to the linear system $$Q' \in |{\cal O}_{\mathbb{P}^1 \times \mathbb{P}^1} (4, 4)| = |- 2 K_{\mathbb{P}^1 \times \mathbb{P}^1}|,$$ thus it has arithmetic genus $1 + K_{\mathbb{P}^1 \times \mathbb{P}^1}^2 = 9$. But $Q'$ has also $9$ double points: indeed, let $p_i, \overline p_i$ be the pre - images of $q_i$, for $i = 1, 2, 3,$ and let $\hat p_i$ be the pre-images of the $x_i$'s. Since the $q_i$'s are double points for $Q$, the points $p_i, \overline p_i$ are double points for $Q'$, and since the $x_i$'s are tangency points between $B$ and $Q$, the $\hat p_i$'s are double points for $Q'$ too. The cover involution $i : \mathbb{P}^1 \times \mathbb{P}^1 \to \mathbb{P}^1 \times \mathbb{P}^1$ swaps the coordinates, it preserves $Q'$, and thus it also preserves the set of its $9$ singular points.\\
 Let $p : X \to \mathbb{P}^1 \times \mathbb{P}^1$ be the blow up of these points, with exceptional divisors $E_i$'s above the $p_i$'s, $\overline E_i$'s above the $\overline p_i$'s, and $\hat E_i$'s above the $\hat p_i$'s. Then the strict transform $C$ of $Q'$ has class $$C = p^* Q' - 2 \sum_i E_i + \overline E_i + \hat E_i = - 2 K_X.$$ We already noted that $Q'$ is a rational curve, thus $C \simeq \mathbb{P}^1$. Finally, the self intersection of $Q'$ equals $$(Q')^2 = (- 2 K_{\mathbb{P}^1 \times \mathbb{P}^1})^2 = 32,$$ thus $$C^2 = 32 + 4 \sum_i E_i^2 + \overline E_i^2 + \hat E_i^2 = 32 - 36 = -4.$$ By Remark \ref{easy}, $X$ is a Coble surface with irreducible boundary curve $C$, with a well - defined involution induced by $(\mathbb{P}^1 \times \mathbb{P}^1, (a, b) \to (b, a))$.\\
This does not contradict the Lemma \ref{nothree}, since $X$ has $(-2)$-curves. Indeed, let $L \subset \mathbb{P}^2$ be the line through $q_1, q_2$. Its preimage $q^* L$ is a smooth curve of type $|{\cal O}_{\mathbb{P}^1 \times \mathbb{P}^1} (1, 1)|$, and it passes through $p_1, p_2, \overline p_1, \overline p_2$. The self intersection of $q^*L$ equals $2 L^2 = 2$. Its strict transform $R \subset X$ has class $$R = p^*(q^* L) - E_1 - E_2 - \overline E_1 - \overline E_2,$$ thus $$R^2 = 2 - 4 = -2.$$ Since $q^* L$ has class $|{\cal O}_{\mathbb{P}^1 \times \mathbb{P}^1} (1, 1)|$, it is isomorphic to $\mathbb{P}^1$, and thus $R$ is isomorphic to $\mathbb{P}^1$ too. Clearly, the same is true picking $L$ as the line through $q_1, q_3$ or $q_2, q_3$.\\
Now consider the point $y_1$, which belongs to the diagonal $\Delta \subset \mathbb{P}^1 \times \mathbb{P}^1$. Let $A \in |{\cal O}_{\mathbb{P}^1 \times \mathbb{P}^1} (1, 0)|, B \in |{\cal O}_{\mathbb{P}^1 \times \mathbb{P}^1} (0, 1)|$ be the two fibers through $y_1$. In the surface $X$, the strict transforms $A', B'$ of $A, B$ are $(-1)$-curves switched by the involution. The set of curves ${E_1, E_2, E_3, \overline E_1, \overline E_2, \overline E_3, A', B', \hat E_2, \hat E_3}$ is invariant with respect to this involution, and its contraction produces another minimal model, namely $\mathbb{P}^2$ equipped with a projective involution. The presence of $(-2)$-curves in $X$ ensures that this does not contradict the Lemma \ref{nofour}.\\
Finally, if $\Delta \subset \mathbb{P}^1 \times \mathbb{P}^1$ denotes the diagonal, then its strict transform $\tilde \Delta \subset X$ is a $(-1)$-curve, since it is obtained from $\Delta$ via the blow up of $y_1, y_2, y_3$. The set of curves ${E_1, E_2, E_3, \overline E_1, \overline E_2, \overline E_3, A', B', \tilde \Delta}$ is invariant with respect to the involution on $X$, and its contraction is the minimal pair $(\mathbb{P}^1 \times \mathbb{P}^1, [X, Y],[Z, W] \to [X, -Y], [Z, -W])$, which is a particular case of case $i)$ of the Theorem \ref{BB}.
\end{example}

\begin{example}
Let $\overline Q \subset \mathbb{P}^2$ be an irreducible quartic curve with a node at a point $p_0$, and let $Q \to \overline Q$ be its normalization. Then $Q$ is a smooth curve of genus $2$, thus its canonical divisor $K_Q$ defines a $g_2^1$ with $6$ ramification points. Assume that the two points $a, b \in Q$ lying over $p_0$ are not ramification points, and assume also that $a + b \ne K_Q$. Assume also that there exists a quadric curve $C \subset \mathbb{P}^2$ which is tri - tangent to $\overline Q$, that is, there are three smooth points of $\overline Q$ where the two curves are tangent, while the remaining two intersection points are distinct. Let $\mathbb{F}_1$ be the first Hirzebruch surface, obtained blowing up $\mathbb{P}^2$ at $p_0$. Then there are natural embeddings $Q, C \subset \mathbb{F}_1$. Let $Y \to \mathbb{F}_1$ be the double cover branched over $Q$. Note that $Q$ is a bisection of $\mathbb{F}_1 \to \mathbb{P}^1$, thus the preimage of a generic fiber $F$ of $\mathbb{F}_1$ is a smooth rational curve. In particular, $Y$ admits a conic fibration over $\mathbb{P}^1$, thus it is a rational surface. Let $L, E$ be the generators of $Pic (\mathbb{F}_1)$, where $L$ is the pullback of a line in $\mathbb{P}^2$, while $E$ is the exceptional divisor over $p_0$. The canonical formula for $K_Y$ gives $$K_Y = q^* (K_{\mathbb{F}_1}) + \frac{1}{2} q^* (Q) = q^* (- L).$$ Meanwhile, the class of $C$ in $\mathbb{F}_1$ is just $2 L$, thus the preimage $\overline C := q^* C$ has class $$\overline C = q^* C = q^* (2 L) = - 2 K_Y.$$ Since $C$ is tritangent to $Q$, $\overline C$ has three double points along the ramification curve of $q$. Let $X$ be the blow up of $Y$ at these $3$ points. The cover involution $i_Y : Y \to Y$ belongs to the Case $ii)$ of the Theorem \ref{BB}, and it lifts to an involution $i_X : X \to X$. Again, $X$ is a Coble surface, with boundary curve given by the strict transform of $\overline C$. This does not contradict the Lemma \ref{notwo}, since $X$ admits $(-2)$-curves. Indeed, let $$R := q^* E \subset Y.$$ Since $\overline Q$ has a double point at $p_0$, $Q$ and $E$ meet at $2$ points, and thus $R$ is still a smooth rational curve, with self intersection $$R = q^* (E^2) = -2.$$ Note that this remains true also in the Coble surface $X$, since we are assuming that the points of tangency between $Q, C$ do not lie on $E$. Thus $R$ is identified with its preimage in $X$, which is still a $(-2)$-curve.
\end{example}

\begin{example}
Consider $\mathbb{P}^1 \times \mathbb{P}^1$, with the involution $i (a, b) := (b, a)$. Let $A_1, A_2, A_3, A_4 \in |{\cal O}_{\mathbb{P}^1 \times \mathbb{P}^1} (1, 0)|$ be four distinct vertical fibers, and let $$B_k := i (A_k)$$ be the corresponding horizontal fibers. Consider the $16$ intersection points $A_i \cap B_j$, and let $p : X \to \mathbb{P}^1 \times \mathbb{P}^1$ be the blow up at these points.\\
The sum $\sum_i A_i + B_i$ belongs to $|{\cal O}_{\mathbb{P}^1 \times \mathbb{P}^1} (4, 4)| = - 2 K_{\mathbb{P}^1 \times \mathbb{P}^1}$. Each of the $8$ fibers $A_i, B_j$ contains $4$ base points, thus their strict transforms have all self intersection $- 4$. By the Remark \ref{easy}, $X$ is a Coble surface, and it admits an involution which is a lift of $(\mathbb{P}^1 \times \mathbb{P}^1, (a, b) \to (b, a))$. Of course this does not contradict the Lemma \ref{nothree}, since the boundary curve of $X$ has $8$ irreducible components.
\end{example}

\begin{example}
Let $C, Q \subset \mathbb{P}^2$ be smooth curves of degree $2, 4$ respectively. Assume that they are totally tangent, so that there are four points $q_1, \ldots, q_4$ such that $$Q \cap C = 2 q_1 + \cdots + 2 q_4.$$Let $$q : Y \to \mathbb{P}^2$$ be the double cover branched over $Q$. The surface $Y$ is a Del Pezzo surface of degree $2$, and the cover involution $i$ is the Geiser involution. The preimage $q^* C$ splits as $$q^*C = A + B,$$ where $A, B$ are smooth rational curves which meet at $4$ points $p_1, \ldots, p_4$. These points are fixed by $i$, thus the blow up $p : X \to Y$ at these points admits a well defined involution $i_X : X \to X$. As usual, we show that $X$ is a Coble surface. Indeed, the canonical formula for $Y$ gives $$K_Y = q^* K_{\mathbb{P}^2} + \frac{1}{2} q^* Q = q^* (- L),$$ where $L$ is the class of a plane line. Thus $$A + B = q^* C = q^* (2 L) = - 2 K_Y.$$ If $E_1, \ldots, E_4 \subset X$ are the exceptional curves over $p_1, \ldots, p_4$, and $\tilde A, \tilde B \subset X$ are the strict transforms of $A, B$ respectively, then $$\tilde A + \tilde B = (p^* A - E_1 - \cdots - E_4) + (p^* B - E_1 - \cdots - E_4) =$$ $$= p^* (A + B) - 2 E_1 - \cdots - 2 E_4 = - 2 p^* K_Y - 2 E_1 \cdots - 2 E_4 = - 2 K_X.$$ Note that the equality $$q^*C = A + B$$ gives also $$(A + B)^2 = 2 C^2 = 8,$$ but $$(A + B)^2 = A^2 + B^2 + 2 A B = 2 A^2 + 8.$$ This implies $$A^2 = B^2 = 0$$ and hence $$\tilde A^2 = \tilde B^2 = -4.$$ By the Remark \ref{easy}, $X$ is a Coble surface, and it admits an involution which has the Geiser involution as minimal model. This clearly is not a contradiction with the Lemma \ref{nofive}, since the boundary curve of $X$ has $2$ irreducible components.
\end{example}

\begin{example}
Let $I : \mathbb{P}^2 \to \mathbb{P}^2$ be a projective involution, and let $p_0, L_0$ be the fixed point and the fixed line of $I$. Consider an irreducible plane cubic ${\cal E} \subset \mathbb{P}^2$, and assume that the node $q$ does not coincide with $p_0$ and it does not lie on $L_0$. Assume also that the intersection ${\cal E} \cap L_0$ is totally transverse. The curves $I ({\cal E})$ has the same properties, and it intersects ${\cal E}$ at $9$ points. Let $p : X \to \mathbb{P}^2$ be the blow up of $\mathbb{P}^2$ at the $11$ points ${\cal E} \cap I ({\cal E}), q, I (q).$ The surface $X$ is a Coble surface with $2$ boundary components, consisting of the strict transforms $\tilde {\cal E}, \tilde{I (\cal E)}$ of ${\cal E}, I (\cal E)$. Indeed, the blowing up of  $q, I (q)$ assures that the strict transforms of the two cubics are smooth rational curves. The $11$ points are double points for ${\cal E} + I (\cal E)$. Let $\Sigma$ be the set of these points: then $$\tilde {\cal E} + \tilde{I ({\cal E})} = q^* ({\cal E} + I ({\cal E})) - 2 \sum_{p \in \Sigma} E_p =$$ $$= q^* (6 L) - 2 \sum_{p \in \Sigma} E_p = - 2 K_X.$$ Since we blow up $9$ smooth points and a double point on each of the two cubics, the self intersection $\tilde {\cal E}^2$ equals $$\tilde {\cal E}^2 = {\cal E}^2 - 9 - 4 = - 4,$$ and the same holds for $\tilde {I ({\cal E})}$. The Remark \ref{easy} states that $X$ is a Coble surface, which is equipped with the lift of a projective involution. This does not contradict the Lemma \ref{nofour}, since the boundary curve of $X$ is reducible.
\end{example}

\phantomsection
\addcontentsline{toc}{section}{References}
\bibliographystyle{plain}


\end{document}